\documentclass[11pt,a4paper,oneside]{amsart}

\usepackage[scale=0.75]{geometry}

\usepackage{amssymb,amsmath}
\usepackage[dvips]{graphics}
\usepackage[dvips]{color}
\usepackage{graphicx}
\usepackage{hyperref}
\usepackage[english]{babel}
\usepackage[T1]{fontenc}       		
\usepackage{mathrsfs}   
\usepackage{comment}
\usepackage[textsize=tiny]{todonotes}
\usepackage{xcolor}

\newtheorem{Theorem}{Theorem}
\newtheorem{Lemma}[Theorem]{Lemma}
\newtheorem{Proposition}[Theorem]{Proposition}
\newtheorem{Corollary}[Theorem]{Corollary}
\newtheorem{Conjecture}[Theorem]{Conjecture}

\theoremstyle{definition}
\newtheorem{Definition}[Theorem]{Definition}
\newtheorem{Example}[Theorem]{Example}
\newtheorem{Remark}[Theorem]{Remark}
\newtheorem{Claim}{Claim}

\providecommand{\ch}[1]{\text{\raise 2pt \hbox{$\chi$}\kern-0.2pt}_{#1}}

\def\XXint#1#2#3{{\setbox0=\hbox{$#1{#2#3}{\int}$}
      \vcenter{\hbox{$#2#3$}}\kern-.5\wd0}}

\newcommand{\N}{\mathbb{N}}     
\newcommand{\R}{\mathbb{R}}     

\newcommand{\calB}{\mathscr{B}}

\newcommand{\calQ}{\mathscr{Q}}
\newcommand{\calR}{\mathscr{R}}
\newcommand{\calS}{\mathscr{S}}

\newcommand{\btheta}{\boldsymbol{\theta}}

\DeclareMathOperator{\diam}{diam}			
\DeclareMathOperator{\loc}{loc}			

\usepackage[mathcal]{eucal}					

\renewcommand{\geq}{\geqslant}
\renewcommand{\leq}{\leqslant}
\renewcommand{\epsilon}{\varepsilon}

\newcommand{\liminfe}{\mathop{\underline{\lim}}}
\newcommand{\limsupe}{\mathop{\overline{\lim}}}

\begin{document}
\author{Emma D'Aniello, Laurent Moonens and Joseph M. Rosenblatt}
\date{January 23, 2018}

\title[Differentiating Orlicz spaces with rare bases of rectangles]{Differentiating Orlicz spaces\\ with rare bases of rectangles}

\begin{abstract}
In the current paper, we study how the speed of convergence of a sequence of angles decreasing to zero influences the possibility of constructing a rare differentiation basis of rectangles in the plane, one side of which makes with the horizontal axis an angle belonging to the given sequence, that differentiates precisely a fixed Orlicz space.
\end{abstract}




     
\maketitle

In the sequel we always call \emph{interval} in $\R^2$ a set of the form $Q=[a,b]\times [c,d]$ where $a<b$ and $c<d$ are real numbers; a \emph{rectangle}, on the other hand, will be any set obtained from an interval by some isometry of the plane.

There is a long history of research around the maximal operator $M_{\btheta}$ associated to a set $\btheta\subseteq [0,2\pi)$ of angles, defined as:
$$
M_{\btheta} f(x):=\sup_{R}\frac{1}{|R|}\int_R |f|,
$$
where the upper bound extends to all rectangles $R$ in $\R^2$ containing $x$, one side of which make an angle $\theta\in\btheta$ with the horizontal axis.

It has been shown in 1977 by A.~Cordoba and R.~Fefferman \cite{CF1977} that whenever $\btheta=\{\theta_j\}$ is the image of a lacunary sequence (\emph{i.e.} satisfying ${\theta_{j+1}}\simeq \lambda {\theta_j}$ for some $0<\lambda<1$), then $M_{\btheta}$ is bounded on $L^2(\R^2)$~---~and hence also on $L^p(\R^2)$ for all $2\leq p<\infty$. This was extended to any $1<p<\infty$ a year later by A.~Nagel, E.~M.~Stein and S.~Wainger in \cite{NSW1978}. When $\btheta=\{0\}\cup\{1/j: j\in\N^*\}$, M.~de Guzman proved in 1981 (see \cite{DEGUZMAN1981}) that $M_{\btheta}$ is always unbounded on $L^p(\R^2)$ for any $p>1$. Later, many authors considered the influence that the ``size'' of the set $\btheta$ has on its yielding the (un)boundedness of $M_{\btheta}$ on some $L^p$ spaces; a case of particular interest was that of Cantor sets, dealt with by \emph{e.g.} N.~Katz in 1996 \cite{KATZ1996} and K.~Hare in 2000 \cite{HARE2000}. Lately, in 2009, M.~Bateman gave in \cite{BATEMAN} a beautiful characterization of sets of angles $\btheta$ yielding the boundedness of $M_{\btheta}$ in $L^p(\R^2)$ for all $1<p<\infty$, showing it is equivalent to the possibility of covering $\btheta$ by a finite collection of $N$-lacunary sets (see \cite{BATEMAN} for a definition). Even more recently in 2013, P.~Hagelstein studied in \cite{HAGELSTEIN} the relation between the Minkowski dimension of $\btheta$ being zero, and the boundedness in $L^p(\R^2)$ of the associated maximal operator; namely he shows that it is a necessary, but not sufficient condition for its boundedness.

{ A first observation we make in the current paper is to observe that the boundedness of $M_{\btheta}$ either occurs in $L^p(\R^2)$ for all $1<p\leq\infty$, or fails to occur in any of those spaces. The fact that one can take $p=\infty$ in the latter statement being, as it seems to us, a new observation following from Bateman's work \cite{BATEMAN}. We expose our short argument in section~\ref{sec.bateman}.}

On another hand, the boundedness on $L^p$ of the maximal operator $M_{\btheta}$ is also related to differentiation properties of the associated differentiation basis; namely $M_{\btheta}$ is bounded on $L^p(\R^2)$ for all $1<p<\infty$ if and only if Lebesgue's differentiation theorem holds in all $L^p(\R^2)$ spaces ($1<p<\infty$) for the associated basis, \emph{i.e.} if and only the following equality holds for all $f\in L^p(\R^2)$ ($1<p<\infty$) and a.e.\ $x\in\R^2$:
$$
f(x)=\lim_{\begin{subarray}{c}R\ni x\\\diam R\to 0\end{subarray}}\frac{1}{|R|}\int_R |f|,
$$
where the limit is taken on rectangles $R$ containing $x$, one side of which makes an angle $\theta\in\btheta$ with the horizontal axis; we then say that the associated basis \emph{differentiates} $L^p(\R^2)$ for all $1<p<\infty$ (see \emph{e.g.} \cite[Chapter~III]{DEGUZMAN1975} where it is studied how the differentiation of function spaces by a basis, and the behavior of the associated maximal operator are related). It is also the result of an observation by J.-O.~Str\"omberg in \cite{STROMBERG1977} that for $\theta=\{2^{-k}:k\in\N\}$, the above differentiation theorem fails in any Orlicz space ``larger'' than $L\log^2 L(\R^2)$, \emph{i.e.} in any $L^\Phi(\R^2)$ where $\Phi$ is an Orlicz function (see below for the precise definition of those term and space) satisfying $\Phi(t)=o(t\log_+^2 t)$ at $\infty$.

In the present paper, we also study for various \emph{sequences} $(\theta_j)$ decreasing to $0$, the differentiation properties of some differentiation bases of rectangles in $\R^2$ whose elements have countably many shapes (\emph{i.e.} for which the ratio between the length of the horizontal and vertical sides belong to a countable set), knowing moreover that one of their sizes make an angle $\theta_j$ with the horizontal axis for some $j$, in terms of the \emph{speed of convergence} of the sequence $(\theta_j)$.

More precisely, we prove three theorems of the following structure, for three classes of sequences $\btheta=(\theta_j)$ decreasing to $0$, where $\Phi$ is some Orlicz function and $L^\Phi(\R^2)$ is the associated Orlicz space (see below for the precise definitions of the terms used here as well as in the following statement).
\begin{Theorem}\label{thm.main}
There exists a countable family $\calQ$ of intervals of the form $[0,L]\times [0,l]$ having the following properties:
\begin{itemize}
\item[(i)] the differentiation basis $\calB$ of all intervals in $\R^2$ whose shape is that of some $Q\in\calQ$, differentiates $L^1(\R^2)$;
\item[(ii)] the differentiation basis $\calB_{\btheta}$ obtained from $\calB$ by allowing its elements to rotate of an angle $\theta\in\btheta$ around their lower left vertex, fails to differentiate $L^\Psi(\R^2)$ for any Orlicz function $\Psi$ satisfying $\Psi=o(\Phi)$ at $\infty$;
\item[(iii)] there exists a differentiation basis $\calB'\subseteq\calB_{\btheta}$ which differentiates $L^\Phi(\R^2)$ but fails to differentiate $L^\Psi(\R^2)$ for any Orlicz function $\Psi$ satisfying $\Psi=o(\Phi)$ at $\infty$ (we then say that $\calB'$ differentiates precisely $L^\Phi(\R^2)$).
\end{itemize}
\end{Theorem}
\noindent Namely, we consider the three following cases in the previous statement:
\begin{itemize}
\item[(1)] $\btheta=(\theta_j)$ satisfying $0<\liminfe_j (\theta_{j+1}/\theta_j)\leq\limsupe_j (\theta_{j+1}/\theta_j)<1$~---~in which case the statement above holds for $\Phi(t)=t(1+\log_+ t)$ and $L^\Phi(\R^2)=L\log L(\R^2)$ (see Theorem~\ref{thm.main1});
\item[(2)] $\btheta=(\theta_j)$ satisfying $0<\liminfe_j (\theta_{j+1}/\theta_j^d)\leq\limsupe_j (\theta_{j+1}/\theta_j^d)<1$ for some integer $d>1$~---~in which case the statement above holds for $\Phi(t)=t(1+\log_+\log_+ t)$ and $L^\Phi(\R^2)=L\log\log L(\R^2)$ (see Theorem~\ref{thm.main2});
\item[(3)] $\btheta=(\theta_j)$ defined by $\theta_j=\arctan(a^{j^d})$ for some $0<a<1$ and $0<d<1$~---~in which case the statement above holds for $\Phi(t)=t(1+\log_+^{1/d} t)$ and $L^\Phi(\R^2)=L\log^{1/d} L(\R^2)$ (see Theorem~\ref{thm.main3}).
\end{itemize}

Those results rely on some geometrical preliminaries (detailed in section \ref{sec.geom}) and on a nice previous work by A.~Stokolos \cite{STOKOLOS1989} concerning differentiation bases of rectangles in relation to the Orlicz spaces they differentiate. More specifically, in \cite{STOKOLOS1989}, A.~Stokolos constructs, for Orlicz spaces ranging (roughly speaking) between $L\log L(\R^2)$ and $L\log^2 L(\R^2)$, differentiation bases of rectangles $\calB$ satisfying the following properties:
\begin{itemize}
\item[(1)] rectangles in $\calB$ have one side forming an angle $2^{-k}$ with the horizontal axis for some $k\in\N$;
\item[(2)] $\calB$ differentiates precisely $L^\Phi(\R^2)$.
\end{itemize}
Our point, in comparison to Stokolos' result, is here to consider more general sequences $(\theta_j)$ decreasing to $0$, and to see how their convergence speed influences the Orlicz space that one can differentiate using rectangles obtained from intervals enjoying given shapes by rotations of some angle $\theta_j$.

The structure of the paper is as follows: after studying some geometrical preliminaries in the spirit of \cite{MOONENS2016}, we obtain in section~\ref{sec.bad} results concerning the Orlicz spaces a differentiation basis of rectangles associated to a given sequence of angles, does \emph{not} differentiate. Combining the results in those two sections with a lemma by A.~Stokolos (see Lemma~\ref{lem.A-Stok} below), we manage in section \ref{sec.comput} to prove the three versions of the Theorem~\ref{thm.main} stated above.

\section{Notations and Definitions}

\subsection{Rectangles.} For our purposes, a \emph{standard interval} in $\R^2$ is an interval of the form $Q=[0,L]\times [0,\ell]$.
We then let $Q_+:=[L/2,L]\times [0,\ell]$. Given another (not necessarily standard) interval $Q'$, we shall say that $Q'$ has the same shape as $Q$ in case there exists $a\in\R^2$ and $\alpha>0$ such that $Q'=a+\alpha Q$~---~calling \emph{shape} of an interval the quotient of its horizontal side by its vertical one, this is equivalent to say their shapes are equal. Given a family of standard intervals $\calQ$, we then denote by $\calB(\calQ)$ the family of all intervals in $\R^2$ having the same shape as some $Q\in\calQ$.

For $\theta\in [0,2\pi)$ we also denote by $r_\theta$ the (counterclockwise) rotation of angle $\theta$ around the origin.
Given a set $\btheta\subseteq [0,2\pi)$, we then denote by $\calB_{\btheta}(\calQ)$ the set of rectangles of the form $r_\theta Q$ for some $Q\in\calB(\calQ)$.

\subsection{Differentiation bases.} A family $\calB=\bigcup_{x\in\R^2}\calB(x)$ of measurable subsets of $\R^2$ with positive measure is called a \emph{differentiation basis} in case the following conditions hold:
\begin{enumerate}
\item[(a)] for each $x\in\R^2$ and each $B\in\calB(x)$, one has $x\in B$;
\item[(b)] for each $x\in\R^2$, one has $\inf\{\diam(B):B\in\calB(x)\}=0$.
\end{enumerate}
It is moreover called \emph{translation invariant} in case one has $\calB(x)=x+\calB(0)$ for all $x\in\R^2$, and homothecy-invariant if it is translation-invariant and if moreover $\alpha \calB(0)=\calB(0)$ for all $\alpha>0$. We shall also say, finally, that the basis $\calB$ enjoys the Buseman-Feller properties in case for any $B\in\calB$, the two properties $B\in\calB(x)$ and $x\in B$ are equivalent~---~from now on, we shall see the collections $\calB(\calQ)$ and $\calB_{\btheta}(\calQ)$ defined above as differentiation bases enjoying the Buseman-Feller property.

Associated to a differentiation basis $\calB$, there is a maximal operator $M_{\calB}$ defined by:
$$
M_\calB f(x):=\sup_{B\in\calB(x)}\frac{1}{|B|} \int_B |f|.
$$
{ When $\btheta\subseteq [0,2\pi)$ is a set of angles, and when $\calB$ is the collection $\calR_{\btheta}$ of all rectangles in the plane, one side of which makes an angle $\theta\in\btheta$ with the horizontal line, we shall briefly denote $M_\calB$ by $M_{\btheta}$.}

One also says that a differentiation basis \emph{differentiates} a function space $X\subseteq L^1_{\loc}(\R^2)$ in case for every $f\in X$, the equality:
$$
f(x)=\lim_{\begin{subarray}{c} \diam B\to 0\\ B\in\calB(x)\end{subarray}} \frac{1}{|B|} \int_B |f|,
$$
holds for a.e.\ $x\in\R^2$.

\subsection{Orlicz spaces.} For our purposes, an \emph{Orlicz function} is a convex, continuous and increasing function $\Phi:[0,\infty)\to [0,\infty)$ satisfying $\Phi(0)=0$ and $\Phi(t)\to\infty$ at $\infty$; we say that an Orlicz function $\Phi$ satisfies the $\Delta_2$ condition in case there is an absolute constant $K>0$ such that one has $\Phi(2t)\leq K\Phi(t)$ for all sufficiently large $t$. The Orlicz function $\Psi:[0,\infty)\to [0,\infty)$ defined by $\Psi(s):=\sup\{t|s|-\Phi(t):0\leq t<\infty\}$ is then called the \emph{complementary function} to $\Phi$ (a general theory of Orlicz spaces is presented the two monographies by M.~A.~Krasnosel'skii and Ya.~B.~Rutickii \cite{KR} and by M.~M.~Rao and Z.~D.~Ren \cite{RR1991}).

Given an Orlicz function $\Phi$, we let $L^\Phi(\R^2)$ 
denote the set of all measurable functions $f$ in $\R^2$ for which $\Phi(|f|)$ is integrable (for $\Phi(t)=t^p$, $1\leq p<\infty$ this yields the usual Lebesgue space 
$L^p(\R^2)$, while for $\Phi(t)=\Phi_{\beta}(t):=t(1+{\log_+}^{\beta}t)$, with $0 <\beta$, and for $\Phi(t):= t(1+\log_+ \log_{+}t)$ we get the Orlicz spaces 
$L\log^{\beta }L(\R^2)$ and $L\log \log L(\R^2)$, respectively). 

Given an Orlicz function $\Phi$, recall that a sublinear operator $T$ is said to be \emph{of weak type 
$(\Phi,\Phi)$} in case there exists a constant $C>0$ such that, for all $f\in L^\Phi(\R^2)$ and all $\alpha>0$, one has:
$$
|\{x\in\R^2:Tf(x)>\alpha\}|\leq\int_{\R^2} \Phi\left(\frac{|f|}{\alpha}\right).
$$
Whenever $\Phi(t)=t^p$ for $p\geq 1$, we shall say that $T$ has \emph{weak type $(p,p)$}. It is a fact that, for a Buseman-Feller homothecy-invariant differentiation basis $\calB$, the following two properties are equivalent for any given Orlicz function $\Phi$:
\begin{enumerate}
\item[(i)] $M_\calB$ is of weak type $(\Phi,\Phi)$;
\item[(ii)] $\calB$ differentiates $L^\Phi(\R^2)$.
\end{enumerate}
The interested reader will find the details of the latter equivalence in \cite[Chapter~III]{DEGUZMAN1975} (see in particular Remark~4, p.~90).

Finally, given an Orlicz function $\Phi$, we shall say that a differentiation basis $\calB$ \emph{differentiates exactly} $L^\Phi(\R^2)$ in cases it differentiates $L^\Phi(\R^2)$ but fails to differentiate $L^\Psi(\R^2)$ for any Orlicz function $\Psi$ satisfying $\Psi(t)=o(\Phi(t))$ at $\infty$.

\subsection{A lemma by A.~Stokolos}
The following useful lemma, which is a particular case of \cite[Lemma~A]{STOKOLOS1989}, will be useful to us in section~\ref{sec.comput}.
\begin{Lemma}\label{lem.A-Stok}
Assume that $\Phi$ is an Orlicz function satisfying the $\Delta_2$ condition and let $\calR=\bigcup_{k\in\N} \calR_k$ where, for each $k\in\N$, $\calR_k$ is a finite collection of rectangles in $\R^2$. Assume also that there exists a sequence $(\lambda_k)$ increasing to $\infty$ as $k\to\infty$, and a sequence of balls $B_k$ satisfying the following properties:
\begin{enumerate}
\item[(i)] all members of $\calR_k$ have equal area;
\item[(ii)] for any finite collection $\calS\subseteq\calR_k$, one has:
$$
\int_{\R^2} \Psi\left(\sum_{R\in\calS} \chi_R\right)\leq c_1\sum_{R\in\calS} |R|,
$$
where $\Psi$ denotes the complementary function to $\Phi$;
\item[(iii)] for any $R\in\calR_k$, one has:
$$
\frac{|R\cap B_k|}{|R|}\geq \frac{c_2}{\lambda_n}\ ;
$$
\item[(iv)] $|\cup\calR_k|\geq c_3\Phi(\lambda_k)|E_k|$;
\end{enumerate}
here, $c_1>0$, $c_2>0$ and $c_3>0$ are constants. Under those assumptions, there exists a differentiation basis $\calB$ satisfying $\calB\subseteq\calB(\calR)$ that differentiates \emph{precisely} $L^\Phi(\R^2)$.
\end{Lemma}

{Let us now move on to the exposition of our results. Before to discuss the type of statements contained in Theorem~\ref{thm.main}, let us first formulate our observation concerning the boundedness in $L^p$, $1<p\leq\infty$, of the maximal operator $M_{\btheta}$.

\section{Boundedness of the directional maximal operator in $L^p$ and $L^\infty$ spaces}\label{sec.bateman}

We announced the forthcoming proposition in the introduction; the new thing here is that one can include $p=\infty$ in its statement.
\begin{Proposition}
Let ${\btheta}$ be as above and let $M_{\btheta}$ be the directional maximal operator associated to $\btheta$ defined above. Then the following dichotomy holds:
\begin{enumerate}
\item Either $M_{\btheta}$ is bounded on $L^p(\R^2)$ for all $1<p< \infty$, in which case the associated differentiation basis $\calB$ differentiates all $L^p(\R^2)$, $1<p\leq\infty$;
\item or $M_{\btheta}$ is unbounded on all $L^p(\R^2)$, $1<p< \infty$; in this case the associated differentiation basis $\calB$ fails to differentiate $L^p(\R^2)$ for all $1<p\leq\infty$.
\end{enumerate}
\end{Proposition}
\begin{proof}
Assuming that $M_{\btheta}$ is unbounded on \emph{some} $L^p(\R^2)$, $1<p<\infty$, it follows from Bateman \cite[Theorem~1, p.~56]{BATEMAN} that ${\btheta}$ admits Kakeya sets, \emph{i.e.} that for each $N\in\N^*$ there exists a collection $\calR_N\subseteq\calR_{\btheta}$ verifying:
\begin{equation}\label{eq.kak}
|\cup\calR_N|\leq \frac 1N \left|\bigcup_{R\in\calR_N} R^*\right|,
\end{equation}
where one denoted by $R^*$ the rectangle having the same center and width as $R$ but three times its length (we hence assume without loss of generality that none of the $R$'s is a square).

We show that $M_{\btheta}$ is unbounded on all $L^p(\R^2)$ by contradiction. Assume thus that there exists some $1<q<\infty$ such that $M_{\btheta}$ is bounded on $L^q(\R^2)$. Since $\calR_{\btheta}$ is a Busemann-Feller differentiation basis invariant with respect to similarities (translations and homothecies), we know, according to \emph{e.g.} de Guzm\'an \cite[p.~90]{DEGUZMAN1975}, that $\calR_{\btheta}$ would then differentiate $L^q(\R^2)$, and hence also $L^\infty(\R^2)$. It would then follow from de Guzm\'an \cite[Theorem~1.2, p.~69]{DEGUZMAN1975} that there exists a constant $c>0$ such that for any bounded measurable set $A\subseteq\R^2$, one has:
\begin{equation}\label{eq.dg}
\left|\left\{M_{\btheta} \chi_A>\frac 14\right\}\right|\leq c|A|.
\end{equation}
Define, for $N\in\N^*$, $A_N:=\cup\calR_N$. Observe that for any $x\in B_N:=\bigcup_{R\in\calR_N} R^*$, there exists some $R_x\in\calR_N\subseteq\calR_{\btheta}$ such that one has $x\in R_x^*$. Now compute:
$$
M_{\btheta} \chi_{A_N}(x)\geq \frac{1}{|R_x^*|}\int_{R_x^*}\chi{A_N}\geq \frac{1}{|R_x^*|}\int_{R_x^*}\chi{R_x}=\frac{|R_x\cap R_x^*|}{|R_x^*|}=\frac 13 >\frac 14.
$$
We hence have $B_N\subseteq \left\{M_{\btheta} \chi_{A_N}>\frac 14\right\}$; using (\ref{eq.kak}) and (\ref{eq.dg}) we then get:
$$
N|A_N|\leq |B_N|\leq c |A_N|,
$$
which is a contradiction for sufficiently large $N$. We conclude that $M_{\btheta}$ is unbounded on all $L^q(\R^2)$, $1<q<\infty$, and that $\calR_{\btheta}$ fails to differentiate $L^\infty(\R^n)$.
\end{proof}
}

We now proceed towards the proof of Theorem~\ref{thm.main} by starting with some geometrical preliminaries that were announced in the introduction.
\section{Some geometrical preliminaries}\label{sec.geom}

The following straightforward geometrical fact is borrowed from \cite{MOONENS2016}.
\begin{Lemma}\label{cl.1} 
Fix real numbers $0\leq\vartheta<\theta<\frac{\pi}{2}$ and $0<2\ell<L$ and let $Q:=[0,L]\times [0,\ell]$. If moreover one has 
$\tan(\theta-\vartheta)\geq 1/\sqrt{\frac 14 \left( \frac{L}{\ell}\right)^2-1}$, 
then $r_\vartheta Q_+$ and $r_\theta Q_+$ are disjoint.
\end{Lemma}

\begin{Lemma}\label{A} Let $\{{\theta}_{j}\} \subset [0, \frac{\pi}{4}]$ be a decreasing sequence such that, letting $m_{j} : = \tan {\theta}_{j}$, we have that there exist 
a constant $ C > 0$, 
a constant $0 < \zeta < 1$ and a sequence $\{t_{k}\}$  so that, for each $k$, for integers $0 \leq j < k$,
$$m_{j} - m_{k} \geq C {\zeta}^{t_{k}}.$$  
Then, there exist constants $d(C) > c(C) >0$ and $e(C)>0$ depending only on $C$ such that, for each $\epsilon>0$ and each integer $k \in {\Bbb N}^{\star} = {\Bbb N} \setminus \{0\}$, 
one can find a standard interval $Q_k=[0,L_k]\times [0,\ell_k]$ and a subset $\btheta_k=(\theta_0,\dots,\theta_k)\subset\btheta$ satisfying $\#\btheta_k =  k+1 $ 
and such that the following hold:
\begin{enumerate}
\item[(i)] $0 < 2 \ell_k < L_k\leq \epsilon$;
\item[(ii)] $ c(C) {\zeta}^{- t_{2k}} \leq\frac{L_k}{\ell_k}\leq d(C) {\zeta}^{- t_{2k}}$;
\item[(iii)] $\left| \bigcup_{\theta\in\btheta_k} r_{\theta} Q_k\right|\geq \frac{k}{2} |Q_k|$.
\item[(iv)] for any subset $\btheta'=(\theta_{i_0},\dots,\theta_{i_l})\subseteq\btheta_k$ ($0\leq l\leq k$, $i_0<i_1<\cdots<i_l$) and any nonnegative, Borel function $\varphi:\R_+\to\R_+$ satisfying $\varphi(0)=0$, one has:
$$
\int_{\R^2} \varphi\left(\sum_{\theta\in\btheta'}\chi_{r_\theta Q_k}\right)\leq e(C) |Q_k| \zeta^{t_{2k}} \sum_{j=0}^{l}\varphi(j+1)\sum_{r=j}^l \zeta^{-t_{i_r}}.
$$
\end{enumerate}
\end{Lemma}

\begin{proof}
To prove this lemma, observe first that letting $m_j:=\tan \theta_j$ for all $j\geq 0$, we have $m_j\leq 1$ for all $j$, 
so that one can compute, for integers $0\leq j<k$:
$$
\tan(\theta_j-\theta_k)=\frac{m_j-m_k}{1+m_jm_k}\geq\frac 12 (m_j-m_k) \geq \frac{C}{2} {\zeta}^{t_{k}}.$$

Now choose, for all $k$, real numbers $0 < 2 l < L \leq \epsilon$ (we write $L$ and $\ell$ instead of $L_k$ and $\ell_k$ here, 
for the index $k$ remains constant all through the proof) satisfying:
$$
\left(\frac{L}{\ell}\right)^2= 4 + 16 C^{-2} \zeta^{-2t_{2k}}.
$$

It is clear that one has:
$$\frac{L}{\ell}= 2 \zeta^{- t_{2k}} \sqrt{\zeta^{2 t_{2k}} + 4 C^{-2}},$$
so that (i) and (ii) hold if we take, for example, $c(C) :=  2 \sqrt{ 4 C^{-2}} = \frac{4}{C}$ and $d(C):= 2 \sqrt{1 + 4 C^{-2}}$. \\
In order to show (iii), define $Q:=[0,L]\times [0,\ell]$ and observe that one has:
\begin{eqnarray*}
\tan (\theta_j-\theta_k) & \geq &  \frac{C}{2} {\zeta}^{t_{2k}} \\
 & = & \frac{C}{2} \frac{2}{C}  \frac{1}{\sqrt{\frac 14 \left(\frac{L}{\ell}\right)^2-1}}\\
& = &  \frac{1}{\sqrt{\frac 14 \left(\frac{L}{\ell}\right)^2-1}},\\
\end{eqnarray*}
for all integers $j<k$ with $k \in {\Bbb N}^{\star}$. According to Lemma~\ref{cl.1}, this ensures that the family $\{r_{\theta_j}Q_+:j\in\N, 0\leq j\leq k\}$ 
consists of pairwise disjoints sets; in particular we get:
$$ \left| \bigcup_{j= 0}^{k} r_{\theta_j} Q \right| \geq \left| \bigsqcup_{j= 0}^{k} r_{\theta_j} Q_+\right| =   k \frac{|Q|}{2},$$
(we used $\sqcup$ to indicate a disjoint union), which proves (iii).

\begin{figure}[t]
\includegraphics[width=7cm]{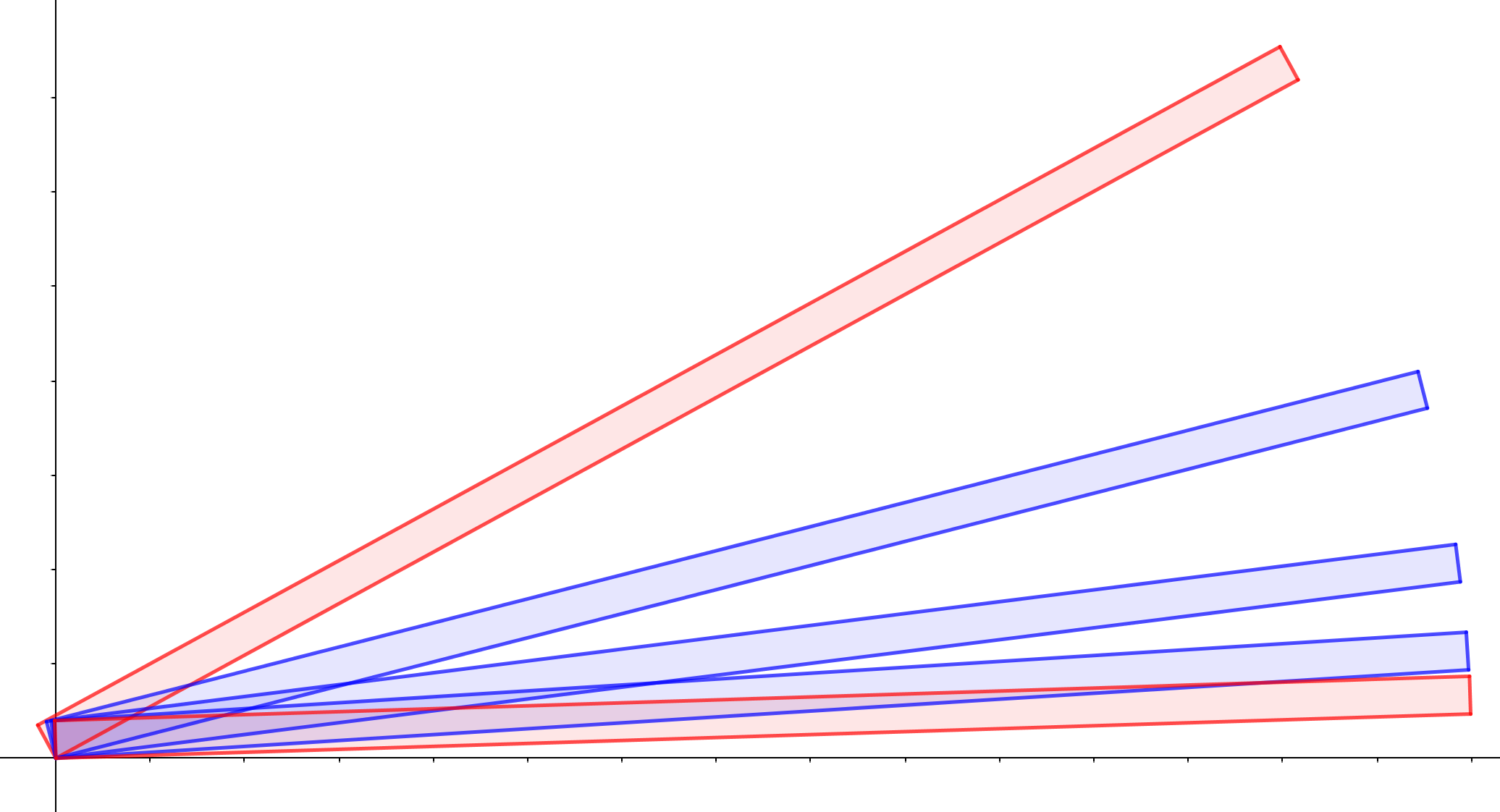}
\includegraphics[width=7cm]{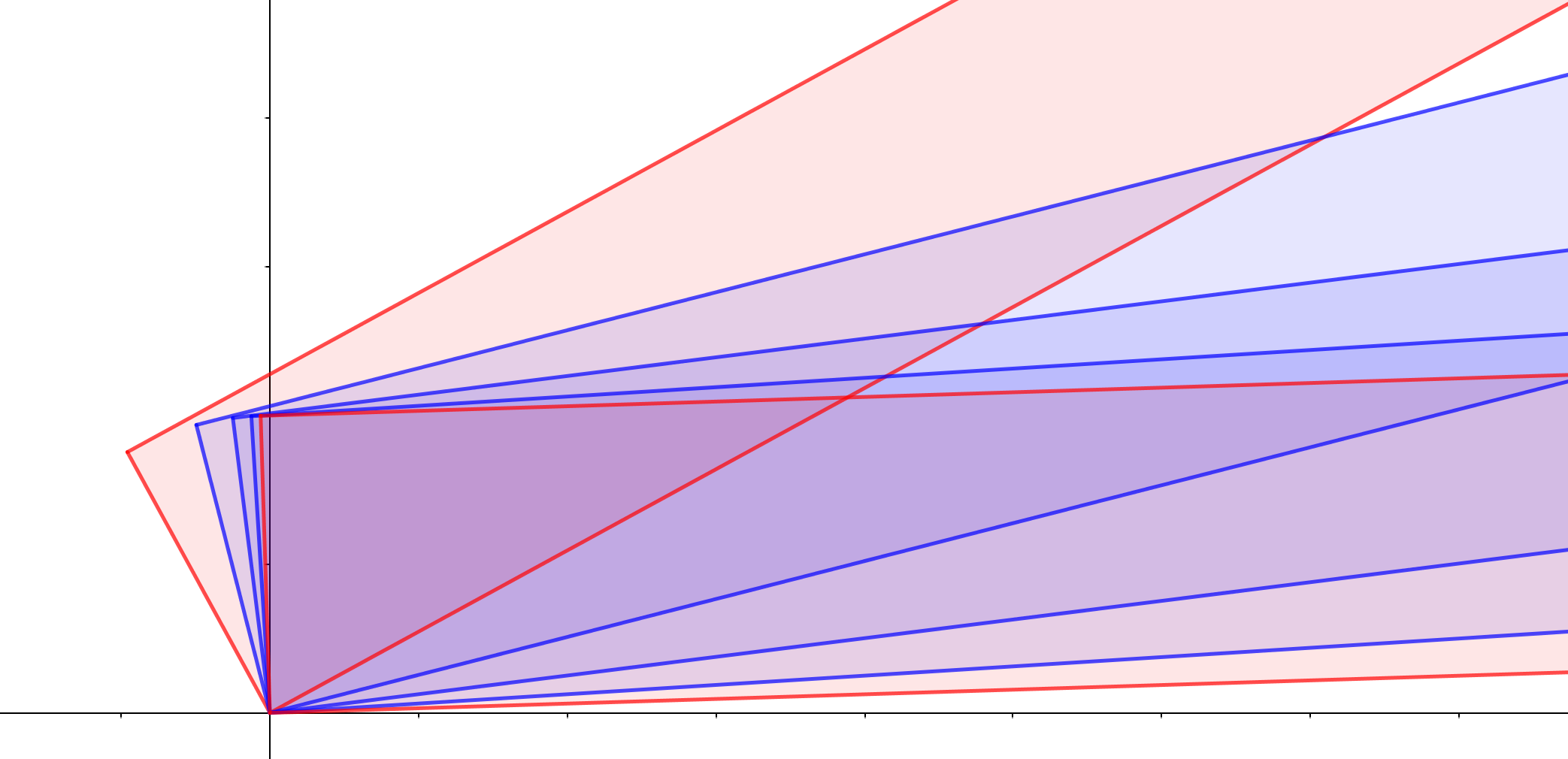}
\caption{Intersection of rectangles $r_{\alpha_{r_i}} Q_k$, $0\leq i\leq p$.}\label{fig.int}
\end{figure}
In order to prove (iv), start by writing $\alpha_r:= \theta_{i_r}$ in order to alleviate notations in the sequel. Given a finite sequence $r_0<\cdots< r_p=$ ($p\leq l$), observe first that one has (see Figure~\ref{fig.int})
$$
\bigcap_{i=0}^p r_{\alpha_{r_i}} Q_k=r_{\alpha_{r_0}} Q_k \cap r_{\alpha_{r_p}} Q_k,
$$
from which it follows that one has (see Figure~\ref{fig.h}):
$$
\left|\bigcap_{i=0}^p r_{\alpha_{r_i}} Q_k\right|=\frac 12 |OP| h\leq \frac 12 \frac{\ell_k}{\sin\hat{P}} \cdot \ell_k\geq \ell_k^2\frac{1}{\tan (\alpha_{r_0}-\alpha_{r_p})},
$$
since one has $\sin\hat{P}=\sin(\alpha_{r_0}-\alpha_{r_p})= \tan{(\alpha_{r_0}-\alpha_{r_p})}\cos{(\alpha_{r_0}-\alpha_{r_p})}\geq\frac 12 \tan(\alpha_{r_0}-\alpha_{r_p})$ (recall that $\alpha_{r_0}-\alpha_{r_p}\leq {\pi}/{4}\leq {\pi}/{3}$.
\begin{figure}[t]
\includegraphics[width=10cm]{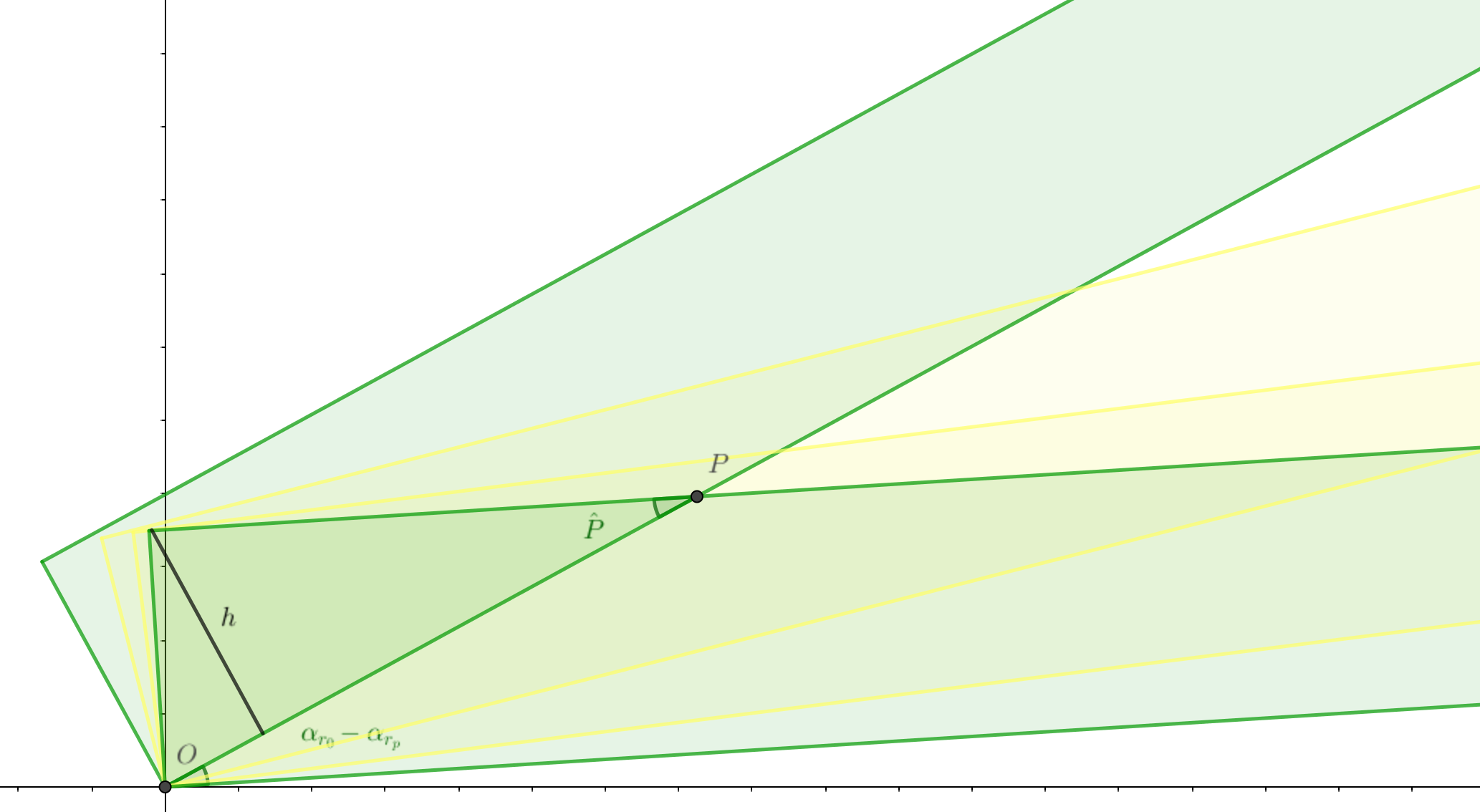}
\caption{Intersection of rectangles $r_{\alpha_{r_i}} Q_k$, $0\leq i\leq p$.}\label{fig.h}
\end{figure}

Now let $\chi:=\sum_{\theta\in\btheta'}\chi_{r_\theta Q_k}$ and fix $0\leq j\leq l$.
\begin{Claim}
One has $|\{\chi=j+1\}|\leq \frac 12 \ell_k^2 \sum_{s=0}^{l-j} \frac{1}{\tan(\alpha_{s}-\alpha_{s+j})}$.
\end{Claim}
\begin{proof}[Proof of the claim]
If $x\in\R^2$ satisfies $\chi(x)=j+1$, there exists $j+1$ angles $\alpha_{r_0},\dots, \alpha_{r_j}$ ($r_0<r_1<\cdots <r_j$) such that one has $$
x\in \bigcap_{i=0}^j r_{\alpha_{r_i}}Q_k=r_{\alpha_{r_0}}Q_k\cap r_{\alpha_{r_j}} Q_k=\bigcap_{r=r_0}^{r_j} r_{\alpha_r}Q_k.
$$
Hence if $r_0<r_1<\cdots <r_j$ is not a sequence of successive integers, we have in fact $\chi(x)\geq 1+r_j-r_0>j+1$.
This proves that one has:
$$
\{\chi=j+1\}\subseteq\bigcup_{s=0}^{l-j} \bigcap_{r=s}^{s+j} r_{\alpha_r}Q_k=\bigcup_{s=0}^{l-j} (r_{\alpha_s}Q_k\cap r_{\alpha_{s+j}}Q_k).
$$
Hence one computes:
$$
|\{\chi=j+1\}|\leq \ell_k^2 \sum_{s=0}^{l-j} \frac{1}{\tan(\alpha_{s}-\alpha_{s+j})},
$$
which proves the claim.
\end{proof}
Now reminding that one has $\tan(\alpha_{s}-\alpha_{s+j})=\tan(\theta_{i_s}-\theta_{i_{s+j}})\geq \frac C2 \zeta^{t_{i_{s+j}}}$ for all $0\leq s\leq l-j$, we compute:
$$
|\{\chi=j+1\}|\leq \frac{2}{C} \ell_k L_k \frac{\ell_k}{L_k} \sum_{s=0}^{l-j} \zeta^{-t_{i_{s+j}}}\leq 
\frac{2}{Cc(C)} |Q_k| \zeta^{t_{2k}} \sum_{r=j}^l \zeta^{-t_{i_r}}.
$$
Letting $e(C):=\frac{2}{Cc(C)}$, we finally obtain:
$$
\int_{\R^2} \varphi(\chi)=\sum_{j=0}^l \varphi(j+1) |\{\chi=j+1\}|\leq e(C)|Q_k|\zeta^{t_{2k}}\sum_{j=0}^l\varphi(j+1) \sum_{r=j}^l \zeta^{-t_{i_r}},
$$
which proves (iv) and hence finishes the proof of the lemma.
\end{proof}

The next proposition will be useful in order to study the maximal operator $M_{r_{\btheta}\calR}$.
\begin{Proposition}\label{B}
Assume that $\btheta$ is as in Lemma \ref{A}.
There exists a family $\calQ=\{Q_k:k\in\N\}$ of standard intervals in $\R^2$ which is totally ordered by inclusion, verifies $\inf\{\diam Q:Q\in\calQ\}=0$ and satisfies the following property: for any sufficiently large integer $k$, there exist sets $\btheta_k\subseteq\btheta$ and $\Theta_k\subseteq\R^2$ satisfying $\#\btheta_k=k+1$ as well as the following conditions (we define $\calR_k:=\{r_\theta Q_k: \theta\in\btheta_k\}$ and $Y_k:=\cup\calR_k$):
\begin{enumerate}
\item[(i)] $|Y_k|\geq \gamma(C) k  \zeta^{- t_{2k}} |\Theta_k|$;
\item[(ii)] for all $R\in\calR_k$, one has:
$$
\frac{|R\cap\Theta_k|}{|R|}\geq   {\gamma}{'}(C)  \zeta^{t_{2k}}\ ;
$$
\item[(iii)] all rectangles in $\calR_k$ have the same area;
\item[(iv)] for any subset $\btheta'=(\theta_{i_0},\dots,\theta_{i_l})\subseteq\btheta_k$ ($0\leq l\leq k$, $0\leq i_0<i_1<\cdots<i_l\leq k$) and any nonnegative, Borel function $\varphi:\R_+\to\R_+$ satisfying $\varphi(0)=0$, one has:
$$
\int_{\R^2} \varphi\left(\sum_{\theta\in\btheta'}\chi_{r_\theta Q_k}\right)\leq \gamma''(C) |Q_k| \zeta^{t_{2k}} \sum_{j=0}^{l}\varphi(j+1)\sum_{r=j}^l \zeta^{-t_{i_r}}.
$$
\end{enumerate}
where $\gamma(C) >0$, ${\gamma}{'}(C)> 0$ and ${\gamma}{''}(C)> 0$ are constants depending only on $C$. 
\end{Proposition}

\begin{proof}
Define $\calR=\{Q_k:k\in\N^*\}$ where the sequence $(Q_k)_{k\in\N^*}$ is defined inductively as follows.
We choose $Q_1=[0,L_1]\times [0,\ell_1]$ and $\btheta_1\subseteq\btheta$ associated to $k=1$ and $\epsilon=1$ according to Lemma~\ref{A}. 
Assuming that $Q_1,\dots, Q_k$ have been constructed, for some integer $k\in\N^*$, we choose $Q_{k+1}=[0,L_{k+1}]\times [0,\ell_{k+1}]$ and $\btheta_{k+1}$ 
associated to $k+1$ and $\epsilon=\min(\ell_{k},1/k)$ according to Lemma~\ref{A}. Since the sequence $(Q_k)_{k\in\N^*}$ is a nonincreasing sequence of rectangles, 
it is clear that $\calQ$ is totally ordered by inclusion. It is also clear by construction that one has $\inf\{\diam Q_k:k\in\N\}=0$.
Now fix $k\in\N^*$ and define $\Theta_k:=B(0,\ell_k)$ and $\calR_k :=\{ r_\theta Q_k:\theta\in\btheta_k\}$. Compute hence, using [Lemma ~\ref{A}, (ii) and (iii)]:
$$
|Y_k|\geq \frac{k}{2} |Q_{k}| =  \frac{k}{2}  L_k l_k = \frac{1}{ 2 \pi} k \frac{L_k}{l_k} |\Theta_k|  \geq \frac{1}{2\pi} k c(C)  \zeta^{- t_{2k}} |\Theta_k|,
$$
so that (i) is proved with $\gamma(C) : = \frac{c(C)}{2 \pi}$.

For any $R\in\calR_k$, there exists $\theta\in\btheta_k$ such that one has $R=r_\theta Q_k$; we hence compute
$$
\frac{|R\cap\Theta_k|}{|R|}=\frac{|\Theta_k\cap r_\theta Q_k|}{|Q_k|}=\frac{\frac 14 \cdot \pi\ell_k^2}{L_k\ell_k}=\frac{\pi}{4} \cdot\frac{\ell_k}{L_k} 
\geq \frac{\pi}{4} \frac{1}{d(C)}  \zeta^{t_{2k}},
$$
which finishes the proof of (ii) if we set ${\gamma}^{'}(C) := \frac{\pi}{4} \frac{1}{d(C)}$. Now (iii) is clear, while (iv) results immediately from [Lemma~\ref{A}, (iv)].
\end{proof}
\begin{Remark}\label{rmq.prop3}
In the conditions of the previous proposition, statement (ii) can be reformulated as follows:
\begin{enumerate}
\item[(ii')] for all $x\in Y_k$, one has $M_{\calB_{\theta}(\calQ)} \chi_{\Theta_k}(x)\geq {\gamma}{'}(C)  \zeta^{t_{2k}}$.
\end{enumerate}
\end{Remark}

Using the previous proposition, we can, using standard techniques developed \emph{e.g.} in a previous work by the second and third authors \cite{MR2012} or in a paper by the current authors \cite{DANMOON},
obtain negative differentiation results in a range of Orlicz spaces for some differentiation bases of rectangles associated to various sets $\btheta$.

\section{Bad Orlicz spaces for some maximal functions}\label{sec.bad}

In the following statement, we let $\Phi_\beta(t):=t(1+\log_+^\beta t)$.
\begin{Proposition}\label{C}
Assume that $\btheta$ is as in Lemma \ref{A}, that the sequence $t_{k}$ tends to $+ \infty$ as $k \rightarrow \infty$ and the sequence 
$\{\frac{ {t_{k}}^{\beta}}{k}\}$ is  bounded above for some $\beta>0$, that is $M: = \limsup_{k}  \frac{ {t_{k}}^{\beta}}{k}<\infty$. There exists a (countable) family $\calQ$ of standard intervals in $\R^2$ with $\inf\{\diam Q:Q\in\calR\}=0$, satisfying the following conditions:
\begin{enumerate}
\item[(i)] $M_{\calB(\calQ)}$ has weak type $(1,1)$, and hence the differentiation basis $\calB(\calQ)$ differentiates $L^1(\R^2)$;
\item[(ii)] for any Orlicz function $\Psi$ satisfying $\Psi=o({\Phi}_{\beta})$  at $\infty$, $M_{\calB_{\btheta}(\calQ)}$ fails to be of weak type $(\Psi,\Psi)$, and hence the associated differentiation basis $\calB_{\btheta}$ fails to differentiate $L^\Psi(\R^2)$.
\end{enumerate}
\end{Proposition}
\begin{proof}
We keep the notations of Proposition~\ref{B} and call $\calQ$ the family of rectangles given by Proposition~\ref{B}. Observe first that, since $\calQ$ is 
totally ordered by inclusion, it follows \emph{e.g.} from \cite[Claim~1]{STOKOLOS2005} 
that $M_{\calB(\calQ)}$ satisfies a weak $(1,1)$ inequality.

In order to show (ii), define, for $k$ sufficiently large, $f_k:= \frac{1}{{\gamma}^{'}(C)} \cdot \zeta^{- t_{2k}} \chi_{\Theta_k}$, where $\Theta_k$ and $Y_k$ 
are associated to $k$ and $\calQ$ according to Proposition~\ref{B}.
\begin{Claim}\label{cl.Mphi} For each sufficiently large $k$, we have:
$$
|\{x\in\R^2:M_{\calB_{\btheta}(\calQ)} f_k(x)\geq 1\}| \geq {\gamma}_{1}(\beta, C, \zeta, M) \int_{\R^2} \Phi_{\beta}(f_k).
$$
\end{Claim}
\begin{proof}[Proof of the claim]
To prove this claim, one observes that for $x\in Y_k$ we have $M_{\calB_{\btheta}(\calQ)} f_k(x)\geq 1$ according to (ii') in Remark~\ref{rmq.prop3}. Yet, on the other hand, one computes, 
for $k$ sufficiently large:
\begin{eqnarray*}
\int_{\R^2} \Phi_{\beta}(f_k) & \leq & 
 \frac{1}{{\gamma}^{'}(C)} \cdot \zeta^{- t_{2k}} |\Theta_k| {\left( 2 t_{2k} \log \frac{1}{\zeta}\right)}^{\beta} \\
& \leq &  \frac{1}{\gamma(C) {\gamma}^{'}(C)}  |Y_k|  (2k)^{-1}  {\left( 2 t_{2k} \log \frac{1}{\zeta}\right)}^{\beta}\\
& < &  2^{\beta + 1} \tilde{M} \frac{1}{\gamma(C){\gamma}^{'}(C)}  |Y_k|   { \left(\log \frac{1}{\zeta}\right)}^{\beta}, 
\end{eqnarray*}
where $\tilde{M} : = \max \{1, M\}$ and $${\gamma}_{1}(\beta, C, \zeta, M) :=  
{ \left(\log \frac{1}{\zeta}\right)}^{- \beta} \left[ \frac{{{\gamma}^{'}(C)}{{\gamma}(C)}}{2^{\beta + 1} \tilde{M}}\right].$$
The claim follows.
\end{proof}
\begin{Claim}
For any $\Phi$ satisfying $\Phi=o(\Phi_{\beta})$ at $\infty$ and for each constant $ T >0$, we have:
$$
\lim_{k\to\infty}\frac{\int_{\R^2} \Phi_{\beta}(|f_k|)}{\int_{\R^2} \Phi( T |f_k|)}=\infty.
$$
\end{Claim}
\begin{proof}[Proof of the claim]
Compute for any $k$:
\begin{eqnarray*}
\frac{\int_{\R^2} \Phi(T|f_k|)}{\int_{\R^2} \Phi_{\beta}(|f_k|)} & = & \frac{\Phi(\zeta^{- t_{2k}}T/{\gamma}'(C))}{\Phi_{\beta}(\zeta^{- t_{2k}}/ {\gamma}^{'}(C))}\\
& = & \frac{\Phi( \zeta^{- t_{2k}}T/{\gamma}^{'}(C))}{\Phi_{\beta}(\zeta^{- t_{2k}} T/ {\gamma}^{'}(C))}
 \frac{\Phi_{\beta}( \zeta^{- t_{2k}} T/ {\gamma}^{'}(C))}{\Phi_{\beta}(\zeta^{- t_{2k}}/{\gamma}^{'}(C))},
\end{eqnarray*}
observe that the quotient $ \frac{\Phi_{\beta}( \zeta^{- t_{2k}}T/{\gamma}^{'}(C))}{\Phi_{\beta}(\zeta^{- t_{2k}}/{\gamma}^{'}(C))}$ is 
bounded as $k\to\infty$ by a constant independent of $k$, while by assumption the quotient 
$\frac{\Phi( \zeta^{- t_{2k}}T/{\gamma}^{'}(C))}{\Phi_{\beta}(\zeta^{- t_{2k}}T/{\gamma}^{'}(C))}$ tends to zero as $k\to\infty$. The claim is proved.
\end{proof}

We now finish the proof of Proposition~\ref{C}. To this purpose, fix $\Phi$ an Orlicz function satisfying $\Phi=o(\Phi_{\beta})$ at $\infty$ 
and assume that there exists a constant $T>0$ such that, for any $\alpha>0$, one has:
$$
|\{x\in\R^2:M_{\calB_{\btheta}(\calQ)} f(x)>\alpha\}|\leq\int_{\R^2} \Phi\left(\frac{T|f|}{\alpha}\right).
$$
Using Claim~\ref{cl.Mphi}, we would then get, for each $k$ sufficiently large:
$$
0< {\gamma}_{1}( \beta, C, \zeta, M) \int_{\R^2}\Phi_{\beta}(f_k)\leq \left|\left\{x\in\R^2:M_{\calB_{\btheta}(\calQ)} f_k(x)> \frac 12 \right\}\right|\leq \int_{\R^n} \Phi({2Tf_k}),
$$
contradicting the previous claim and proving the theorem.\end{proof}

The following result is proved in a very similar way~---~reason for which we here omit its straightforward proof.\begin{Proposition}\label{D}
Assume that $\btheta$ is as in Lemma \ref{A}, that the sequence $t_{k}$ tends to $+ \infty$ as $k \rightarrow \infty$, 
that it satisfies $M:= \limsup_{k} \frac{\log t_{k}}{k} < \infty$ and define $\Phi(t) = t ( 1 + \log_+ \log_+ t)$. 
There exists a (countable) family $\calQ$ of standard intervals in $\R^2$ with $\inf\{\diam Q:Q\in\calQ\}=0$, satisfying the following conditions:
\begin{enumerate}
\item[(i)] $M_{\calB(\calQ)}$ has weak type $(1,1)$, and hence the associated differentiation basis $\calB(\calQ)$ differentiates $L^1(\R^2)$;
\item[(ii)] for any Orlicz function $\Psi$ satisfying $\Psi=o({\Phi})$  at $\infty$, $M_{\calB_{\btheta}(\calQ)}$ fails to be of weak type $(\Psi,\Psi)$~---~hence the differentiation basis $\calB_{\btheta}(\calQ)$ fails to differentiate $L\log \log L(\R^2)$.
\end{enumerate}
\end{Proposition}

We now turn to proving the three versions of Theorem~\ref{thm.main} we stated in the introduction.

\section{Three examples of sequences yielding Theorem~\ref{thm.main}}\label{sec.comput}
Let $d \in {\Bbb N}$ be fixed.  Assume that the sequence $(\theta_k)_{k\in\N}\subseteq (0,\pi/4]$ is such that one has:
\begin{equation}\label{eq.bilac0}
0<\lambda<\liminfe_{j\to\infty} \frac{\theta_{j+1}}{{\theta_j}^{d}}\leq \limsupe_{j\to\infty} \frac{\theta_{j+1}}{{\theta_j}^{d}}<\mu<1.
\end{equation} 

Letting $m_j:=\tan \theta_j$ for all $j\in\N$, one clearly has:
$$
\lim_{j\to\infty}\frac{m_j}{\theta_j}=1,
$$
so that (\ref{eq.bilac0}) also holds for the sequence $(m_j)_{j\in\N}$. There hence exists an index $j_0\in\N$ such that, for all $j\geq j_0$, one has 
$\lambda\leq\frac{m_{j+1}}{{m_j}^{d}}\leq \mu$ (we may also and will assume that one has $\frac{\lambda}{2} \leq m_{j_0}  \leq {\lambda}$). 
For the sake of clarity, we shall now consider that $j_0=0$ and compute, for an integer $0\leq j< k$:
$$
\tan(\theta_j-\theta_{k})=\frac{m_j-m_{k}}{1+m_jm_{k}}\geq\frac 12 (m_j-m_{k}).
$$
We also have, for every integer $0\leq j< k$:
\begin{equation}\label{eq.gen-d} {\lambda}^{\sum_{i=0}^{k-j-1} d^{i}} {m_j}^{d^{k-j}}\leq m_{k} \leq \mu^{\sum_{i=0}^{k-j-1} d^{i}} {m_j}^{d^{k-j}}. \end{equation}

\subsection{Assume first that $d=1$.} Arguing as in \cite{MOONENS2016}, we then get:
$$\tan(\theta_j-\theta_{k})  \geq   \frac 12 {{m}_0} \cdot   {\lambda}^{k} \cdot [\mu^{- 1} - 1],$$
so that the if we set $C=   \frac{1}{2} {{m}_0} \cdot   [\mu^{- 1} - 1]$, $\zeta = \lambda$, $t_{k} = k$ and take $\beta= 1$ then the hypotheses of Proposition~\ref{C} are satisfied and let $\calQ$, $\calR_k$ and $\btheta_k$ be associated to $\btheta$ by Proposition~\ref{B}; define $\Phi(t):=t(1+\log_+ t)$ and observe that it is easy to see (see \emph{e.g.} \cite[Chapter~1]{KR}) that one can have $\Psi(t)\leq K_1 e^t$
for the complementary function $\Psi$ to $\Phi$. There is no loss of generality, of course, to assume $\zeta<1/e$.

Now fix a subset $\btheta'=\{\theta_{i_r}:0\leq r\leq l\}\subset \btheta_k$ ($0\leq i_0<i_1<\cdots <i_l\leq k$) and write, using $\varphi=\Psi$ in [Proposition~\ref{B}, (iv)]:
$$
\int_{\R^2} \Psi\left(\sum_{\theta\in\btheta'}\chi_{r_\theta Q_k}\right)\leq K_1' |Q_k| \zeta^{2k} \sum_{j=0}^{l}e^j\sum_{r=j}^l \zeta^{-i_r},
$$
where $K_1'=eK_1e(C)>0$. Write then:
$$
\sum_{r=j}^l \zeta^{-i_r}\leq \zeta^{-i_l} \sum_{s=0}^\infty \zeta^s\leq \frac{\zeta^{-k}}{1-\zeta}.
$$
On the other hand one computes:
$$
\sum_{j=0}^l e^j=e^l\sum_{j=0}^l e^{j-l}\leq {e^k}\sum_{s=0}^\infty e^{-s}\leq \frac{e^{k+1}}{e-1}.
$$
Hence we obtain:
$$
\int_{\R^2} \Psi\left(\sum_{\theta\in\btheta'}\chi_{r_\theta Q_k}\right)\leq K_1' |Q_k| \zeta^{2k} \cdot \frac{e^{k+1}}{e-1}
\cdot \frac{\zeta^{-k}}{1-\zeta}\leq \frac{2K_1'}{1-\zeta} |Q_k|,
$$
since one has $e\zeta<1$ and $e/(e-1)\leq 2$.

Combining now Proposition~\ref{C} and Lemma~\ref{lem.A-Stok} for the particular sequence considered in this section, we get the following result.
\begin{Theorem}\label{thm.main1}
Assume that the sequence $(\theta_k)_{k\in\N}\subseteq (0,\pi/4]$ is such that one has:
$$
0<\lambda<\liminfe_{j\to\infty} \frac{\theta_{j+1}}{{\theta_j}}\leq \limsupe_{j\to\infty} \frac{\theta_{j+1}}{{\theta_j}}<\mu<1.
$$
Under those assumptions, there exists a countable family of standard intervals in $\R^2$ denoted by $\calQ$ and totally ordered by inclusion, satisfying the following properties (where one defines $\Phi(t):=t(1+\log_+t)$):
\begin{itemize}
\item[(i)] the associated differentiation basis $\calB(\calQ)$ differentiates $L^1(\R^2)$;
\item[(ii)] for any Orlicz function $\Psi$ satisfying $\Psi=o(\Phi)$ at $\infty$, the differentiation basis $\calB_{\btheta}(\calQ)$ fails to differentiate $L^\Psi(\R^2)$;
\item[(iii)] there exists a differentiation basis $\calB\subseteq\calB_{\btheta}(\calQ)$ that differentiates exactly $L\log L(\R^2)$.
\end{itemize}
\end{Theorem}
\begin{Remark}
Statement (iii) above is weaker than stating that the basis $\calB_{\btheta}(\calQ)$ itself does differentiate (exactly) $L\log L(\R^2)$; however, it is not known to us whether this statement holds or not.
\end{Remark}

\subsection{Assume now that $d >1$ in (\ref{eq.bilac0}).} Using (\ref{eq.gen-d}), we obtain for $0\leq j<k$:
$$
{\lambda}^{\frac{d^{k -j} -1}{d-1}}  {m_j}^{d^{k-j}} = {\lambda}^{\sum_{i=0}^{k-j-1} d^{i}} {m_j}^{d^{k-j}}\leq m_{k} \leq \mu^{\sum_{i=0}^{k-j-1} d^{i}} {m_j}^{d^{k-j}} = 
{\mu}^{\frac{d^{k -j}-1}{d-1}}  {m_j}^{d^{k-j}}.$$
We then compute, for the same $j<k$:
\begin{eqnarray*}
\tan(\theta_j-\theta_{k}) & \geq & \frac 12 (m_j-m_{k})\geq \frac 12 [{{m}_k}^{\frac{1}{d^{k-j}}} \cdot \mu^{- \frac{\sum_{i=0}^{k-j-1} d^{i}}{d^{k- j}}} - {m}_{k}]\\
&  \geq &  \frac 12 {{m}_k} \cdot [{\mu}^{- \frac{\sum_{i=0}^{k-j-1} d^{i}}{d^{k- j}}}  - 1]\\
& \geq &  \frac 12 {{m}_0}^{d^{k}} \cdot   {\lambda}^{\sum_{i=0}^{k-1} d^{i}} \cdot [\mu^{- \frac{1}{d-1}} - 1]\\
& = &  \frac 12 {{m}_0}^{d^{k}} \cdot   {\lambda}^{d^{k}} \cdot [\mu^{-   \frac{1}{d-1}} - 1]\\
& \geq & \frac 12 {{m}_0}^{d^{k}} \cdot   {\lambda}^{d^{k}} \cdot [\mu^{-  \frac{1}{d-1}} - 1] \\
& \geq &  \frac 12 {(\frac{\lambda}{2})}^{d^{k}} \cdot   {\lambda}^{d^{k}} \cdot [\mu^{-   \frac{1}{d-1}} - 1] \\
& \geq & \frac 12 {(\frac{\lambda}{2})}^{2d^{k}} \cdot [\mu^{-  \frac{1}{d-1}} - 1] .\\
\end{eqnarray*}
So if we set $C=  \frac 12 [\mu^{-   \frac{1}{d-1}} - 1]$, $\zeta = {(\frac{\lambda}{2})}^{2}$ and $t_{k} = d^{k}$,  then the hypotheses of Proposition~\ref{D} are satisfied. Let hence $\calQ$, $\calR_k$ and $\btheta_k$ be associated to $\btheta$ by Proposition~\ref{B} and define $\Phi(t):=t(1+\log_+\log_+ t)$ and observe that it is easy to see (see \emph{e.g.} \cite[Chapter~1]{KR}) that one can have $\Psi(t)\leq K_2 \exp(\exp t)$
for the complementary function $\Psi$ to $\Phi$. There is no loss of generality, again, to assume $\zeta<1/e$.

Fix as before a subset $\btheta'=\{\theta_{i_r}:0\leq r\leq l\}\subset \btheta_k$ ($0\leq i_0<i_1<\cdots <i_l\leq k$) and write, using $\varphi=\Psi$ in [Proposition~\ref{B}, (iv)]:
$$
\int_{\R^2} \Psi\left(\sum_{\theta\in\btheta'}\chi_{r_\theta Q_k}\right)\leq K_2' |Q_k| \zeta^{d^{2k}} \sum_{j=0}^{l}e^{e^{j+1}}\sum_{r=j}^l \zeta^{-d^{i_r}},
$$
where $K_2'=K_2e(C)>0$. Write then:
$$
\sum_{r=j}^l \zeta^{-d^{i_r}}\leq \zeta^{-d^{i_l}} \sum_{s=0}^\infty \zeta^s\leq \frac{\zeta^{-d^k}}{1-\zeta}.
$$
On the other hand one computes:
$$
\sum_{j=0}^l e^{e^{j+1}}\leq e^{e^{l+1}}\sum_{s=0}^\infty e^{-s}\leq \frac{e^{1+e^{k+1}}}{e-1}.
$$
Hence we obtain:
$$
\int_{\R^2} \Psi\left(\sum_{\theta\in\btheta'}\chi_{r_\theta Q_k}\right)\leq K_2'e |Q_k| \zeta^{d^{2k}} \cdot \frac{e^{e^{k+1}}}{e-1}
\cdot \frac{\zeta^{-d^k}}{1-\zeta}\leq \frac{K_2'e}{(1-\zeta)(e-1)} |Q_k| \exp\left[e^{k+1}+d^k-\left(d^2\right)^k\right],
$$
since one has $e\zeta<1$. Observing \emph{e.g.} that $d^2\geq 4>e$ and that one hence has: $$
e^{k+1}+d^k-\left(d^2\right)^k=e\left[e^k-\frac{1}{2e} \left(d^2\right)^k\right]+ d^k-\frac 12 \left(d^2\right)^k,
$$
and since moreover it is clear that $a^k-\epsilon b^k$ tends to $-\infty$ as $k\to\infty$ for any real numbers $1<a<b$ and $\epsilon>0$, we finally get:
$$
\lim_{k\to\infty} \exp\left[e^{k+1}+d^k-\left(d^2\right)^k\right] = 0.
$$
There thus exists a constant $K_2''>0$ (depending only on $\btheta$ and $d$) for which one has:
$$
\int_{\R^2} \Psi\left(\sum_{\theta\in\btheta'}\chi_{r_\theta Q_k}\right)\leq K_2'' |Q_k|.
$$

Combining now as above Proposition~\ref{D} and Lemma~\ref{lem.A-Stok} for the particular case of a sequence satisfying (\ref{eq.bilac0}) for $d>1$, we obtain the following result.
\begin{Theorem}\label{thm.main2}
Assume that $d\in\N^*$ and the sequence $(\theta_k)_{k\in\N}\subseteq (0,\pi/4]$ are such that one has:
$$
0<\lambda<\liminfe_{j\to\infty} \frac{\theta_{j+1}}{{\theta_j^d}}\leq \limsupe_{j\to\infty} \frac{\theta_{j+1}}{{\theta_j^d}}<\mu<1.
$$
Under those assumptions, there exists a countable family of standard intervals in $\R^2$ denoted by $\calQ$ and totally ordered by inclusion, satisfying the following properties (where one defines $\Phi(t):=t(1+\log_+\log_+t)$):
\begin{itemize}
\item[(i)] the associated differentiation basis $\calB(\calQ)$ differentiates $L^1(\R^2)$;
\item[(ii)] for any Orlicz function $\Psi$ satisfying $\Psi=o(\Phi)$ at $\infty$, the differentiation basis $\calB_{\btheta}(\calQ)$ fails to differentiate $L^\Psi(\R^2)$;
\item[(iii)] there exists a differentiation basis $\calB\subseteq\calB_{\btheta}(\calQ)$ that differentiates exactly $L\log \log L(\R^2)$.
\end{itemize}
\end{Theorem}

\subsection{A non-lacunary example}
Let $(a_{j})$ be a nonincreasing sequence of positive real numbers with $0< a = \inf_{j} a_{j} \leq \sup_{j} a_{j}= b <1$. Clearly, $a_{0} = \max a_{j} =b$.  
In the sequel we fix a real number $0<d<1$ and we define a sequence $(\theta_j)$ by:
$$
\theta_j:=\arctan \left[{(a_{j})}^{j^d}\right].
$$
Observe that letting $m_j:=\tan \theta_j$ for all $j$, we can write, for $j_0\leq j\leq k$:
\begin{eqnarray}
m_j-m_k & = &  {(a_{j})}^{j^d} - {(a_{k})}^{k^d}  \geq  {(a_{k})}^{j^d} - {(a_{k})}^{k^d}  \nonumber\\
& = &    {(a_{k})}^{k^d}\left[ {(a_{k})}^{j^d - k^{d}} - 1\right]\nonumber\\
&  \geq &  {a}^{k^d}\left[ {\left(\frac{1}{b}\right)}^{k^d - j^{d}} - 1\right]  \nonumber\\
& \geq &  {a}^{k^d}\left[ {\left(\frac{1}{b}\right)}^{k^d - ({k-1})^{d}} - 1\right].\label{eq.aaa}
\end{eqnarray}
It hence follows in particular that $(m_j)$ (and hence also $(\theta_j)$) is a decreasing sequence.

On the other hand, it is easy to observe that one has, for all $j$:
$$
j^d-(j-1)^d=dj^{d-1}+O(j^{d-2}),
$$
we get:
$$
\left[j^{d}-(j-1)^d\right] j^{1-d}=d+O(j^{-1}),
$$
so that, for $j$ sufficiently large, we have:
\begin{equation}\label{eq.jd}
j^d-(j-1)^d\geq \frac {d}{2 j^{1-d}}.
\end{equation}
It is also the case that for $j$ sufficiently large, we always have:
\begin{equation}\label{eq.jexp}
\frac 2d j^{1-d} \leq \frac 12 b^{-j^d}\ ;
\end{equation}
we shall hence assume that both (\ref{eq.jd}) and (\ref{eq.jexp}) hold for $j\geq j_0$, and we shall, from now on, work with the sequence $(\theta_j)_{j\geq j_0}$; we also define $\btheta:=\{\theta_j:j\geq j_0\}$. 

Now given $k>j_0$, we obtain from (\ref{eq.aaa}), for all $\max \{ j_0 ,  \log_a (\log_a e)\} \leq  j_{1} \leq j \leq k$ , using the inequality $a^x\geq 1+x$ 
for all $x \geq \log_a (\log_a e)$ and then (\ref{eq.jd}) and (\ref{eq.jexp}):
$$
m_j-m_k  \geq {a}^{k^d}\left[ {\left(\frac{1}{b}\right)}^{k^d - ({k-1})^{d}} - 1\right]  \geq   {a}^{k^d}  \frac{d}{2 k^{1-d}} \geq 2 {(a b )}^{k^d}.
$$
So if we take $C= 2$,  $\zeta = ab$, $t_{k} = k^{d}$ and $\beta = \frac 1d$, we have that the hypotheses of Proposition~\ref{C} are satisfied. Let also $\calQ$, $\calR_k$ and $\btheta_k$ be associated to $\btheta$ by Proposition~\ref{B}; define now $\Phi(t):=t(1+\log_+^{1/d} t)$ and observe that it is easy to see (see \emph{e.g.} \cite[Chapter~1]{KR}) that one can have $\Psi(t)\leq K_3 e^{t^d}$
for the complementary function $\Psi$ to $\Phi$. There is no loss of generality, reducing $\zeta$ if necessary, to assume that one has $\eta:=e\zeta^{2^d-1}<1$.

Now fix a subset $\btheta'=\{\theta_{i_r}:0\leq r\leq l\}\subset \btheta_k$ ($0\leq i_0<i_1<\cdots <i_l\leq k$) and write, using $\varphi=\Psi$ in [Proposition~\ref{B}, (iv)]:
$$
\int_{\R^2} \Psi\left(\sum_{\theta\in\btheta'}\chi_{r_\theta Q_k}\right)\leq K_3' |Q_k| \zeta^{(2k)^d} \sum_{j=0}^{l}e^{(j+1)^d}\sum_{r=j}^l \zeta^{-i_r^d},
$$
where $K_3'=K_3e(C)>0$. Write then:
$$
\sum_{r=j}^l \zeta^{-i_r^d}\leq \zeta^{-i_l^d} \sum_{r=j}^l \zeta^{i_l^d-i_r^d}\leq (k+1)\zeta^{-k^d},
$$
since we have $\zeta^{i_l^d-i_r^d}\leq 1$ for all $j\leq r\leq l$.
On the other hand one computes in a similar fashion:
$$
\sum_{j=0}^l e^{j^d}=e^{l^d}\sum_{j=0}^l e^{j^d-l^d}\leq (k+1){e^{k^d}}.
$$
Hence we obtain:
$$
\int_{\R^2} \Psi\left(\sum_{\theta\in\btheta'}\chi_{r_\theta Q_k}\right)\leq K_3' |Q_k| (k+1)^2\zeta^{(2k)^d-k^d} e^{k^d}.
$$
Yet one has:
$$
(k+1)^2\zeta^{(2k)^d-k^d} e^{k^d}=(k+1)^2 \left(e\zeta^{2^d-1}\right)^{k^d}=(k+1)^2 \eta^{k^d}\to 0,\quad k\to\infty.
$$
Hence there exists a constant $K_3''>0$, independent of $k$, for which one has:
$$
\int_{\R^2} \Psi\left(\sum_{\theta\in\btheta'}\chi_{r_\theta Q_k}\right)\leq K_3'' |Q_k|.
$$

Combining now Proposition~\ref{C} and Lemma~\ref{lem.A-Stok} for the particular sequence considered in this section, we get the following result.
\begin{Theorem}\label{thm.main3}
Assume that $0<d<1$ is a real number.
Let $(a_{j})$ be a nonincreasing sequence of positive real numbers with $0<  \inf_{j} a_{j} \leq \sup_{j} a_{j} <1$ and define a sequence $(\theta_j)$ by:
$$
\theta_j:=\arctan \left[{(a_{j})}^{j^d}\right].
$$
Under those assumptions, there exists a countable family of standard intervals in $\R^2$ denoted by $\calQ$ and totally ordered by inclusion, satisfying the following properties (we define $\Phi(t):=t(1+\log_+^{1/d} t)$ for $\beta>0$):
\begin{itemize}
\item[(i)] the associated differentiation basis $\calB(\calR)$ differentiates $L^1(\R^2)$;
\item[(ii)] for any Orlicz function $\Psi$ satisfying $\Psi=o(\Phi)$ at $\infty$, the differentiation basis $\calB_{\btheta}(\calQ)$ fails to differentiate $L^\Psi(\R^2)$;
\item[(iii)] there exists a differentiation basis $\calB\subseteq\calB_{\btheta}(\calQ)$ that differentiates exactly $L\log^{\frac 1d} L(\R^2)$~---~hence it also differentiates $L^p(\R^2)$ for all $p>1$.
\end{itemize}
\end{Theorem}
\begin{Remark} Let $(a_{j})$ and $d$ be as above, that is let $\{a_{j}\}$ be a non-increasing sequence of positive real numbers with 
$0< a = \inf_{j} a_{j} \leq \sup_{j} a_{j}= b <1$ and $0< d< 1$. Then ${(a_{j})}^{j^d}$ cannot be lacunary according to the definition in 
\cite{MOONENS2016}. Assume, by contradiction that there exist $\alpha$ and $\beta$ satisfying
$$0 < \alpha \leq \frac{{a_{j+1}}^{{(j+1)}^d}}{{a_{j}}^{{j}^d} }< \beta < 1.$$
Then, for each $j \geq 2$, it is
$${\alpha}^{\frac{j-1}{j^{d}}} {a_{1}}^{\frac{1}{j^{d}}} \leq a_{j} \leq {\beta}^{\frac{j-1}{j^{d}}} {a_{1}}^{\frac{1}{j^{d}}},$$
so that it must be $\lim_{j} {a_{j}} =0$, and this contradicts the hypothesis $\inf_{j} a_{j}   = \lim_{j} a_{j} = a >0$.  \end{Remark}

\subsection*{Acknowledgements.} Both authors would like to thank their respective research institutes for the warm hospitality they could enjoy from them when preparing the present manuscript. The second author would like to acknowledge the support of the ``Laboratory Ypatia of Mathematical Sciences'', thanks to which the current collaboration has been possible.

\vspace{0.3cm}
\noindent
\small{\textsc{Emma D'Aniello},}
\small{\textsc{Dipartimento di Matematica e Fisica},}
\small{\textsc{Scuola Politecnica e delle Scien\-ze di Base},}
\small{\textsc{Universit\`a degli Studi della Campania ``Luigi Vanvitelli''},}
\small{\textsc{Viale Lincoln n. 5, 81100 Caserta,}}
\small{\textsc{Italia}}\\
\footnotesize{\texttt{emma.daniello@unicampania.it}.}

\vspace{0.3cm}
\noindent
\small{\textsc{Laurent Moonens},}
\small{\textsc{Laboratoire de Math\'ematiques d'Orsay, Universit\'e Paris-Sud, CNRS UMR8628, Universit\'e Paris-Saclay},}
\small{\textsc{B\^atiment 307},}
\small{\textsc{F-91405 Orsay Cedex},}
\small{\textsc{Fran\-ce}}\\
\footnotesize{\texttt{laurent.moonens@math.u-psud.fr}.}

\vspace{0.3cm}
\noindent
\small{\textsc{Joseph M. Rosenblatt},}
\small{\textsc{Department of Mathematical Sciences, Indiana University-Purdue University Indianapolis},}
\small{\textsc{402 North Blackford Street},}
\small{\textsc{Indianapolis, IN 46202-3216},}
\small{\textsc{U.S.A.}}\\
\footnotesize{\texttt{joserose@iupui.edu}.}


\end{document}